\documentclass[reqno,11pt]{amsart}
\usepackage{relsize}
\usepackage{scalerel}
\usepackage[margin=1in]{geometry}
\usepackage{stackengine,wasysym}
\usepackage{todonotes}
\usepackage{xcolor}
\usepackage{mathrsfs}
\usepackage{dsfont}
\usepackage{mathtools}
\mathtoolsset{showonlyrefs}
\usepackage[hyperfootnotes=false,colorlinks=true,linkcolor = blue, urlcolor  = blue, citecolor = blue]{hyperref}
\usepackage[sort,nocompress]{cite}
\usepackage{float}
\usepackage{amsmath, amsthm, amssymb}
\usepackage{times}
\usepackage{color}
\usepackage{comment}
\usepackage[toc,page]{appendix}
\usepackage{verbatim}
\usepackage{enumerate}

\newcommand{\pa}{\partial}
\newcommand{\la}{\label}
\newcommand{\fr}{\frac}
\newcommand{\na}{\nabla}
\newcommand{\be}{\begin{equation}}
\newcommand{\bal}{\begin{aligned}}
\newcommand{\eal}{\end{aligned}}
\newcommand{\ee}{\end{equation}}
\newcommand{\ba}{\begin{array}{l}}
\newcommand{\ea}{\end{array}}

\newcommand{\beg}{\begin}

\renewcommand{\div}{{\mbox{div}\,}}
\newcommand{\D}{\Delta}

\newtheorem{thm}{Theorem}

\usepackage{MnSymbol,wasysym}

\newcommand{\R}{\mathbb R}

\theoremstyle{plain}

\newtheorem{Thm}{Theorem}[section]

\newtheorem{lem}[Thm]{Lemma}
\newtheorem{prop}[Thm]{Proposition}

\theoremstyle{definition}

\theoremstyle{remark}
\newtheorem{rem}{Remark}

\renewcommand{\div}{{\mbox{div}\,}}
\numberwithin{equation}{section}

\title{Logarithmic Sobolev Inequalities for Bounded Domains and Applications to Drift-Diffusion Equations}

\author[E. Abdo]{Elie Abdo}
\address[E. Abdo]
{	Department of Mathematics \\
     University of California  \\
	Santa Barbara, CA 93106-3080, USA.} \email{elieabdo@ucsb.edu}

\author[F.-N. Lee] {Fizay-Noah Lee}
\address[F.-N. Lee] 
{Department of Mathematics\\
Vanderbilt University\\
Nashville, TN 37235, USA. }
\email{noah.lee@vanderbilt.edu}

\begin{document}

\begin{abstract}
We prove logarithmic Sobolev inequalities on higher-dimensional bounded smooth domains based on novel Gagliardo-Nirenberg type interpolation inequalities. Moreover, we use them to address the long-time dynamics of some nonlinear nonlocal drift-diffusion models and prove the exponential decay of their solutions to constant steady states. 
\end{abstract} 


\maketitle
\tableofcontents

\section{Introduction}

Logarithmic Sobolev inequalities, initiated by Stam in \cite{S}, have been extensively studied in the literature over the last five decades due to their relevance in quantum field theory, probability theory, and geometry. Such inequalities have been initially studied for Gaussian measures \cite{AC, G, S} and then generalized to Euclidean measures (see for instance \cite[Chapter 8]{LL}, \cite{NS}, and references therein).

In this paper, we derive Logarithmic Sobolev inequalities for bounded domains, and we show how they can serve as a tool to study the long-time dynamics of some nonlinear nonlocal drift-diffusion equations on $d$-dimensional bounded connected domains with smooth boundaries. More precisely, given such a domain $\Omega \subset \R^d$ for $2 \le d \le 6$, and an integer $p \ge 3$ obeying $(d-2)p \le 2d$, we prove that 
\be \la{lsi}
\int_{\Omega} g \ln \frac{g}{\bar{g}} dx
\le A_p {\bar{g}}^{\frac{p-4}{p-2}}\left\|\na \left(g^{\fr{1}{p-2}}\right)\right\|_{L^2}^2, \hspace{1cm} \bar{g} = \fr{1}{|\Omega|} \int_{\Omega} g(x) dx
\ee holds for nonnegative scalar functions $g \in H^1(\Omega)$, where $A_p$ is a positive constant depending only on $p$, $d$, and the diameter of $\Omega$.  Our proof is based on the following three main ingredients:
\begin{enumerate}
    \item Novel Gagliardo-Nirenberg type interpolation inequalities of the form
    \be \la{gni}
\int_{\Omega} q(x)^p dx
\le \bar{q}^2 \|q\|_{L^{p-2}}^{p-2} + \sum\limits_{i=1}^{p} C_i \bar{q}^{\alpha_i} \| q - \bar{q}\|_{L^2}^{\beta_i} \|\na q\|_{L^2}^{\gamma_i}, \hspace{1cm} \bar{q} =  \frac{1}{|\Omega|}\int_{\Omega} q(x) dx,
    \ee that hold for nonnegative functions $q$ and integers $p \ge 2$ obeying $(d-2)p \le 2d$. The nonnegative constants $C_i$ depend only on $p$ and the dimension $d$, and the nonnegative numbers $\alpha_i, \beta_i, \gamma_i$  satisfy $\alpha_i + \beta_i + \gamma_i = p$ and $\beta_i + \gamma_i \ge 2$ for any $i \in \left\{1, \dots, p \right\}$. The 
    term $\bar{q}^2 \|q\|_{L^{p-2}}^{p-2}$ on the right-hand side of \eqref{gni} is sharp and is crucial for the derivation of the desired estimates \eqref{lsi}.
    \item The concavity of the logarithm function, which allows for the application of Jensen's inequality and reduces the task of establishing \eqref{lsi} to a problem of seeking good control of some $L^p$ norm. 
    \item Basic logarithmic estimates 
    \be 
\ln \left(1 + \sum\limits_{i=1}^{n} C_i x^i \right) \le Cx
    \ee that hold for any $x \ge 0$ and some positive universal constant $C=C(C_1,\dots,C_n)$.  
\end{enumerate}

One of the main challenges is the derivation of \eqref{gni}, which implies (from applications of the classical $L^2$ Poincar\'e inequality) that the $p$-th power of the $L^p$ norm of $q$ is controlled by a linear combination of powers of $\|\na q\|_{L^2}$ plus the sharp term $\bar{q}^2 \|q\|_{L^{p-2}}^{p-2}$ (which is simply bounded by $1$ for specific choices of $q$). Additionally, we must ensure that the linear combination of powers of $\|\na q\|_{L^2}$ does \textit{not} include low powers; in particular, the exponents must be at least $2$. To this end, we seek a special decomposition of the $L^p$ norm of $q$, which allows us to make use of standard interpolation inequalities while preserving the aforementioned desired properties.  

This paper is organized as follows. In Section \ref{II}, we derive sharp Gagliardo-Nirenberg type interpolation inequalities \eqref{gni}, which apply to nonnegative functions in $H^1(\Omega)$. Then, in Section \ref{LSI}, we use the interpolation inequalities from Section \ref{II} to derive the logarithmic Sobolev inequalities for bounded domains \eqref{lsi}. Lastly, in Section \ref{app}, we show how the logarithmic Sobolev inequalities from Section \ref{LSI} can be used to prove exponential convergence to equilibrium for the Nernst-Planck system, which is a nonlinear, nonlocal drift-diffusion system. We show that the rate of convergence depends naturally only on the domain and the diffusion coefficients.

\section{Interpolation Inequalities}\la{II}

In this section, we prove a class of sharp Gagliardo-Nirenberg type inequalities. The inequality \eqref{interpolation} is sharp in the sense that if $q$ is taken to be a positive constant $c$, then we have 
\be 
\int_{\Omega} q(x)^p dx = c^p |\Omega| =  c^2 c^{p-2}  |\Omega| = \bar{q}^2 \|q\|_{L^{p-2}}^{p-2}.
\ee Indeed, \eqref{interpolation} gives the exact corrector arising from interpolating a  scalar function with a non-vanishing average. This will be used to establish a logarithmic Sobolev inequality that could serve as a main tool to study the long-time behavior of solutions to a large class of nonlinear nonlocal drift-diffusion equations. 

\beg{prop} \la{inte} Let $d \ge 2$. Let $\Omega \subset \R^d$ be a bounded, connected domain with smooth boundary. If $d=2$, then let $p \ge 2$ be an integer. If $d\ge 3$, then let $p$ be an integer such that $2\le p\le \frac{2d}{d-2}$. Then the following interpolation inequality
\be \la{interpolation}
\beg{aligned}
\int_{\Omega} q(x)^p dx 
&\le \bar{q}^2 \|q\|_{L^{p-2}}^{p-2} 
+ C\|q - \bar{q}\|_{L^2}^{\fr{2d+2p- pd}{2}} \|\na q\|_{L^2}^{\fr{d(p-2)}{2}} + C\bar{q}^{p-2}\|q - \bar{q}\|_{L^2}^{\fr{2d+ 2p - dp}{p}}\|\na q\|_{L^2}^{\fr{d(p-2)}{p}}
\\&+C\sum\limits_{k=3}^{p-1} \bar{q}^{k-2} \|q-\bar{q}\|_{L^2}^{(p-k+2)-\fr{d(p-k)}{2}}\|\na q\|_{L^2}^{\fr{d(p-k)}{2}}
+ C\bar{q}^{p-2}\|\na q\|_{L^2}^2
\end{aligned}
\ee holds for any nonnegative scalar function $q \in H^1(\Omega)$. Here $\bar{q}$ denotes the average of $q$ over $\Omega$, and $C$ is a positive constant that depends only on $d$, $p$ and the diameter of $\Omega$. (We define $\|q\|_{L^0}^0:=|\Omega|$, and we assume the summation is empty if $p\le 3$.)
\end{prop}

\beg{proof}
We decompose the integral $\int_{
\Omega} q^p$ into the sum of three terms, $I_1, I_2$ and $I_3$, where
\beg{align}
I_1 &= \int_{\Omega} (q - \bar{q})^2 q^{p-2} dx,
\\I_2 &= -\bar{q}^2 \int_{\Omega} q^{p-2} dx, 
\\I_3 &= 2\bar{q} \int_{\Omega} q^{p-1} dx. 
\end{align}
We apply H\"older's inequality and make use of Sobolev embeddings and Gagliardo-Nirenberg interpolation inequalities to estimate $I_1$ as follows,
\be 
\beg{aligned}
I_1 
&\le \|q - \bar{q}\|_{L^p}^2 \|q^{p-2}\|_{L^{\fr{p}{p-2}}} 
= \|q - \bar{q}\|_{L^p}^2   \|q\|_{L^{p}}^{p-2}
\\&\le C\|q - \bar{q}\|_{L^p}^2 \left(\|q - \bar{q}\|_{L^p}^{p-2} + |\bar{q}|^{p-2} \right)
\\&\le C\|q - \bar{q}\|_{L^2}^{\fr{2p - dp 
+ 2d}{p}}\|\na q\|_{L^2}^{\fr{d(p-2)}{p}}  \left(\|q - \bar{q}\|_{L^2}^{\fr{(p-2)(2p - dp +2d)}{2p}} \|\na q\|_{L^2}^{\fr{d(p-2)^2}{2p}} + |\bar{q}|^{p-2} \right)
\\&\le C\|q - \bar{q}\|_{L^2}^{\fr{2d+2p- pd}{2}} \|\na q\|_{L^2}^{\fr{(p-2)d}{2}} + C|\bar{q}|^{p-2}\|q - \bar{q}\|_{L^2}^{\fr{2p - dp 
+ 2d}{p}}\|\na q\|_{L^2}^{\fr{d(p-2)}{p}}.
\end{aligned}
\ee 
Now we decompose $I_3$ into the sum of four terms, $I_{3,1}, I_{3,2}$, $I_{3,3}$, and $I_{3,4}$, where 
\beg{align}
I_{3,1} &= 2\sum\limits_{k=3}^{p-1} \bar{q}^{k-2} \int_{\Omega} (q - \bar{q})^2 q^{p-k} dx \left(=0 \text{ if } p\le 3\right),
\\I_{3,2} &= 2 \bar{q}^{p-2} \int_{\Omega} (q - \bar{q})^2 dx,
\\I_{3,3} &= 2 \bar{q}^{p-1} \int_{\Omega} (q - \bar{q}) dx,
\\I_{3,4} &= 2\bar{q}^2 \int_{\Omega} q^{p-2} dx.
\end{align} It is evident that the quantity $I_{3,3}$ vanishes. An application of the Poincar\'e inequality to the mean-free function $q - \bar{q}$ yields the bound
\be 
I_{3,2} \le C\bar{q}^{p-2} \|\na q\|_{L^2}^2.
\ee As for the term $I_{3,1}$, (assuming $p>3$) we apply H\"older's inequality, interpolate, and obtain 
\be
\beg{aligned}
I_{3,1} 
&\le C\sum\limits_{k=3}^{p-1} \bar{q}^{k-2}  \int_{\Omega} (q-\bar{q})^2\left((q - \bar{q})^{p-k} + \bar{q}^{p-k}\right)
\\&\le C\sum\limits_{k=3}^{p-1} \bar{q}^{k-2}\|q- \bar{q}\|_{L^{p-k+2}}^{p-k+2} + C\bar{q}^{p-2} \|q - \bar{q}\|_{L^2}^2
\\&\le C\sum\limits_{k=3}^{p-1} \bar{q}^{k-2} \|q-\bar{q}\|_{L^2}^{(p-k+2)-\fr{d(p-k)}{2}}\|\na q\|_{L^2}^{\fr{d(p-k)}{2}}
+ C\bar{q}^{p-2}\|\na q\|_{L^2}^2
\end{aligned}
\ee Putting all these estimates together, we infer that \eqref{interpolation} holds, ending the proof of Proposition \ref{inte}.
\end{proof}


\section{Logarithmic Sobolev Inequalities}\la{LSI}

Prior to proving the logarithmic Sobolev inequalities, we establish the following lemma, which, roughly speaking, gives us linear control of logarithmic terms.

\begin{lem} \la{log}
Let $f:[0,\infty)\to [0,\infty)$ be a $C^1$ 
function satisfying
\begin{align}
f(0)&=0\\
\fr{f'(x)}{1+f(x)}&\le A\quad\text{for some constant $A$, for all } x\ge 0.\la{growth}
\end{align}
Then $\ln(1+f(x))\le Ax$ for any $x\ge 0.$ In particular, for any $\alpha_1,...,\alpha_n\ge 1$, it holds that
\begin{align}
    \ln\left(1+\sum_{i=1}^n C_ix^{\alpha_i}\right)\le Ax
\end{align}
for any $x \ge 0$, where $A$ is a positive constant depending on the nonnegative constants $C_1,\dots,C_n$ and $\alpha_1,\dots,\alpha_n$.
\end{lem}
\begin{proof}
Define $G(x) = \ln(1+f(x))-Ax$. We have $G(0)=0$. And,
\begin{align}
G'(x) = \fr{f'(x)}{1+f(x)}-A\le 0,
\end{align}
so it follows that $G(x)\le 0$ for all $x\ge 0$, and the lemma follows.

\end{proof}


\beg{thm}\la{logsobprop}
Let $2\le d \le 6$ be an integer. Let $p$ be an integer such that $p \in [3, \fr{2d}{d-2}]$ if $d > 2$ and $p \in [3, \infty)$ if $d=2$. Let $\Omega \subset \R^d$ be a bounded, connected domain with smooth boundary. Let $g \in H^1(\Omega)$ be a nonnegative scalar function. Then there exists a positive constant $A_p$ depending only on $p$, $d$, and the diameter of $\Omega$ such that 
\be \la{logsob}
\int_{\Omega} g \ln \frac{g}{\bar{g}} dx
\le A_p {\bar{g}}^{\frac{p-4}{p-2}}\left\|\na \left(g^{\fr{1}{p-2}}\right)\right\|_{L^2}^2.
\ee holds.
\end{thm}

\beg{proof}
As $\frac{g}{|\Omega| \bar{g}} dx$ is a probability measure and $\ln x$ is a concave function, it follows from Jensen's inequality that 
\be \la{relent}
\beg{aligned}
&\int_{\Omega} g \ln \frac{g}{\bar{g}} dx
= \fr{(p-2)|\Omega| \bar{g}}{2} \int_{\Omega} \frac{g}{|\Omega| \bar{g}} \ln \left(\frac{g}{\bar{g}}\right)^{\fr{2}{p-2}} dx
\\&\le \fr{(p-2)|\Omega| \bar{g}}{2} \ln \left(\int_{\Omega}  \frac{g^{\fr{p}{p-2}}}{|\Omega| \bar{g}^{\fr{p}{p-2}}} dx \right)
= \fr{p-2}{2}\|g\|_{L^1} \ln \left(\left\|\fr{1}{|\Omega|^{\fr{1}{p}}} \left(\fr{g}{\bar{g}} \right)^{\fr{1}{p-2}} \right\|_{L^p}^p \right).
\end{aligned}
\ee We let $q = \fr{1}{|\Omega|^{\fr{1}{p}}} \left(\fr{g}{\bar{g}} \right)^{\fr{1}{p-2}}$. A straightforward application of H\"older's inequality with exponents $p-2$ and $\fr{p-2}{p-3}$ ($=\infty$ if $p=3$) yields the estimate
\be \la{qbar}
\bar{q}^2 = \left(\fr{1}{|\Omega|} \int_{\Omega} \fr{1}{|\Omega|^{\fr{1}{p}}} \left(\fr{g}{\bar{g}} \right)^{\fr{1}{p-2}}  dx \right)^2
\le \fr{1}{|\Omega|^{2 + \fr{2}{p}}} \left(\int_{\Omega} \fr{g}{\bar{g}} dx \right)^{\fr{2}{p-2}} |\Omega|^{\fr{2p-6}{p-2}} = |\Omega|^{-2-\frac{2}{p}}|\Omega|^\frac{2}{p-2} |\Omega|^{\frac{2p-6}{p-2}} = |\Omega|^{-\fr{2}{p}},
\ee and a direct computation gives the equality
\be 
\|q\|_{L^{p-2}}^{p-2} = \int_{\Omega} \frac{g}{\bar{g}|\Omega|^{\fr{p-2}{p}}} dx = |\Omega|^{\fr{2}{p}}.
\ee 
Applying Proposition \ref{inte}, the Poincar\'e inequality $\|q- \bar{q}\|_{L^2} \le C\|\na q\|_{L^2}$, and estimate \eqref{qbar}, we deduce that 
\be 
\|q\|_{L^p}^p \le 1 +  C_1 \|\na q\|_{L^2}^2  + C_2\|\na q\|_{L^2}^p + \sum\limits_{k=3}^{p-1}C'_k \|\na q\|_{L^2}^{p-k+2}
\ee 
holds, where $C_1, C_2$ and the $C'_k$'s depend only on $\Omega$ and $p$. As before, the summation is considered empty if $p\le 3$.

Since $\na q = \frac{\na g^{\fr{1}{p-2}}}{|\Omega|^{\fr{1}{p}} \bar{g}^{\fr{1}{p-2}}}$, it follows from Lemma \ref{log} that
\be 
\ln (\|q\|_{L^p}^p) \le  \frac{A'_p}{\bar{g}^\frac{2}{p-2}}\left\|\na \left(g^{\fr{1}{p-2}}\right)\right\|_{L^2}^2,
\ee where $A'_p$ is a positive constant depending only on $p$, $d$, and the diameter of $\Omega$. Returning to \eqref{relent}, we obtain the desired estimate \eqref{logsob}.
\end{proof}

\begin{rem}\la{rem}
    We note that the inequality \eqref{logsob} becomes
    \be\la{specialcase}
    \int_\Omega g\ln\frac{g}{\bar{g}}\,dx\le A_4\|\na \sqrt{g}\|_{L^2}^2
    \ee
    when $d=2,3,4$ and $p=4$. We can compare this inequality to Gross's logarithmic Sobolev inequality \cite{G}, which in particular for $0\le f\in H^1(\mathbb{R}^d)$ says
    \be\la{gross}
    \int_{\mathbb{R}^d}f\ln \frac{f}{\bar{f}}\,d\nu\le C\int_{\mathbb{R}^d}|\na \sqrt{f}|^2\,d\nu
    \ee
    where $\nu$ is the standard Gaussian measure on $\mathbb{R}^d$ and 
    $$\bar{f} = \int_{\mathbb{R}^d}f\,d\nu.$$
    The Bakry-Emery criteria together with the perturbation lemma of Holley and Strook give conditions on probability measures $\mu$ on $\mathbb{R}^d$ for which \eqref{gross} holds with $\nu$ replaced by $\mu$ (see \cite{AMTU} and references therein). In addition, by making using of a convex cutoff function, one can deduce from \eqref{gross} a Poincar\'e-type inequality 
    \be\la{gross2}
    \int_{D}f\ln \frac{f}{\bar{f}}\,d\mu\le C\int_{D}|\na \sqrt{f}|^2\,d\mu
    \ee
    for uniformly convex domains $D$ and admissible probability measures $\mu$ on $D$. Essentially, the inequality \eqref{specialcase} treats the special case of the measure $|\Omega|^{-1}dx$ for $d=2,3,4$ and at least in this context, generalizes \eqref{gross2} to domains that need not be convex. Let us note, however, that the original proofs of \eqref{gross} and \eqref{gross2} follow very different approaches from ours.
\end{rem}

\section{Application: Long-Time Behavior of Solutions to Electrodiffusion Models}\la{app}

The Nernst-Planck (NP) system is a mathematical model for the electrodiffusion of charged particles and is given by
\begin{align}
    \pa_t c_i&= D_i \div (\na c_i+z_ic_i\na\Phi),\quad i=1,\dots,N\la{np}\\
    -\epsilon\D\Phi&=\rho = \sum_{i=1}^Nz_ic_i \la{pois}
\end{align}
The variable $c_i$ represents the local ionic concentration of the $i$-th ionic species and its time evolution is modeled by the Nernst-Planck equations \eqref{np}. The dynamics of $c_i$ are dictated by diffusion due to its own concentration gradient and by the presence of an electrical field, $\na\Phi$. The electrical potential $\Phi$, in turn, is generated by the charge density $\rho$ via a Poisson equation. 

The $D_i>0$ are the constant diffusion coefficients, $\epsilon>0$ represents the dielectric permittivity of the medium, and the $z_i$ are the ionic valences, which may take on positive or negative values. We note that $\epsilon$ is proportional to the square of the Debye length and, in most physical regimes, is a very small constant. This fact motivates the study of the NP system in the context of singular perturbation theory where $\epsilon$ is considered in the limit of $\epsilon\to 0$ \cite{EN,QN}. 

We consider the above system on a bounded domain $\Omega \subset\mathbb{R}^d$ for $d\ge 2$ with smooth boundary. A wide range of physically relevant boundary conditions may be considered (insert ref). Here, we take homogeneous Neumann boundary conditions for $\Phi$ and all $c_i$:
\be
\bal\la{BC}
    {\pa_n c_i}_{|\pa\Omega}&=0,\quad i=1,\dots,N\\
    {\pa_n \Phi}_{|\pa\Omega}&=0.
\eal
\ee
This choice of boundary conditions necessitates a few comments. First, strictly speaking $\Phi$ is determined uniquely only up to a constant from \eqref{pois} and \eqref{BC}. Thus for concreteness, we impose the additional condition that
\begin{align}\la{phimeanfree}
\int_\Omega \Phi\,dx=0
\end{align}
for all time. 

Second, ionic concentration is conserved in the sense that
\begin{align}
    \fr{d}{dt}\int_\Omega c_i\,dx=0\la{conservation}
\end{align}
for each $i$. We note, formally integrating \eqref{pois}, we find that the net charge density must be zero for all time
\begin{align}\la{zerorho}
\int_\Omega \rho\,dx=0\quad t\ge 0 .
\end{align} 
Therefore, to be self consistent, we restrict ourselves to initial conditions such that $\int_\Omega \rho(0)\,dx=0$. Incidentally, this also requires us to have both positive and negative valences $z_i$.

For our choice of boundary conditions, it is well known that the following dissipative energy equality holds:
\begin{align}\la{energy}
    \fr{d}{dt}\mathcal{E}+\mathcal{D}= 0   
\end{align}
where
\begin{align}\la{E}
\mathcal{E} = \sum_{i=1}^N\int_\Omega c_i\ln\fr{c_i}{\bar{c}_i}\,dx+\fr{\epsilon}{2}\int_\Omega |\na\Phi|^2\,dx
\end{align}
where $\bar{c}_i=\fr{1}{|\Omega|}\int_\Omega c_i\,dx$, and
\begin{align}\la{D}
\mathcal{D}=\sum_{i=1}^ND_i\int_\Omega c_i|\na(\ln c_i+z_i\Phi)|^2\,dx.     
\end{align}

For the purpose of this paper, we do not directly address the issue of global regularity of solutions and instead \textit{assume} that we are working with globally regular solutions. We could instead study global weak solutions, in which case the energy equality \eqref{energy} is instead an inequality; however we shall not pursue this path. For the state of the art of global wellposedness and regualrity, we refer the reader to \cite{ci, np3d, cil}.

For regular solutions, the derivation of the above energy equality follows straightforwardly from multiplying \eqref{np} by the quantity $\ln c_i+z_i\Phi$, summing in $i$, and integrating by parts. Indeed, we have
\be
\bal
\sum_{i=1}^ND_i\int_\Omega\div(\na c_i+z_ic_i\na\Phi)(\ln c_i+z_i\Phi)\,dx &=-\sum_{i=1}^ND_i\int_\Omega(\na c_i+z_ic_i\na\Phi)\cdot\na(\ln c_i+z_i\Phi)\,dx \\
&=-\sum_{i=1}^ND_i\int_\Omega c_i|\na(\ln c_i+z_i\Phi)|^2\,dx 
\eal
\ee
where in the first equality, no boundary terms arise due to the boundary conditions \eqref{BC}. On the other hand, on the left hand side of \eqref{np}, we get

\be
\bal
\sum_{i=1}^N\int_\Omega \pa_tc _i(\ln c_i+z_i\Phi)\,dx&=\sum_{i=1}^N\int_\Omega \pa_t (c_i\ln c_i-c_i)+\int_\Omega (\pa_t \rho)\Phi\,dx\\
&=\sum_{i=1}^N\fr{d}{dt}\int_\Omega c_i\ln c_i\,dx-\epsilon\int_\Omega (\pa_t \D\Phi)\Phi\,dx\\
&=\sum_{i=1}^N\fr{d}{dt}\int_\Omega c_i\ln c_i\,dx+\epsilon\int_\Omega (\pa_t \na\Phi)\cdot\na\Phi\,dx\\
&=\sum_{i=1}^N\fr{d}{dt}\int_\Omega c_i\ln c_i\,dx+\fr{\epsilon}{2}\fr{d}{dt}\int_\Omega |\na\Phi|^2\,dx
\eal
\ee
where in the second line we used \eqref{conservation}. Then, we observe that due to \eqref{conservation} again, we have $\fr{d}{dt}\int_\Omega c_i\ln\fr{c_i}{\bar{c}_i}\,dx = \fr{d}{dt}\int_\Omega c_i\ln c_i\,dx$, and this completes the derivation of \eqref{energy}.

We note that the main reason we define $\mathcal{E}$ with $\int_\Omega c_i\ln\fr{c_i}{\bar{c}_i}\,dx$ instead of $\int_\Omega c_i\ln c_i\,dx$ is that (as we will see below) $c_i\to \bar{c}_i$ in the time asymptotic limit, and thus we view $\int_\Omega c_i\ln\fr{c_i}{\bar{c}_i}\,dx$ as an expression of relative entropy, which in a way measures the distance between $c_i$ and $\bar{c}_i$.

For our choice of boundary conditions \eqref{BC}, we have the following lower bound for $\mathcal{D}$
\be
\bal
\mathcal{D}&\ge D\sum_{i=1}^N\int_\Omega c_i|\na(\ln c_i+z_i\Phi)|^2\,dx\\
&=D\sum_{i=1}^N\int_\Omega c_i\left|\fr{\na c_i}{c_i}+z_i\na\Phi\right|^2\,dx\\
&= D\sum_{i=1}^N\int_\Omega \fr{|\na c_i|^2}{c_i}+2z_i\na c_i\cdot\na\Phi+z_i^2c_i|\na \Phi|^2\,dx\\
&= D\sum_{i=1}^N\int_\Omega 4|\na \sqrt{c_i}|^2+2z_i\na c_i\cdot\na\Phi+z_i^2c_i|\na \Phi|^2\,dx\\
&= D\sum_{i=1}^N\int_\Omega 4|\na \sqrt{c_i}|^2+z_i^2c_i|\na \Phi|^2\,dx+\fr{2D}{\epsilon}
\int_\Omega \rho^2\,dx\\
&=:\tilde{\mathcal{D}}
\eal
\ee
where $D=\min_{i}D_i$, and the second to last line follows from integrating by parts using the boundary conditions \eqref{BC} and the Poisson equation \eqref{pois}. In summary, we have shown,

\begin{prop}\la{energyprop}
    Let $c_1,...,c_N$ be a classical solution of \eqref{np}-\eqref{pois} with initial concentrations satisfying boundary conditions \eqref{BC} and net electroneutrality,
    $$\sum_{i=1}^Nz_i\bar{c}_i =0$$
    where $\bar{c}_i=\frac{1}{|\Omega|}\int_\Omega c_i(0)\,dx.$ Then $c_1,...,c_N$ satisfy
    \begin{align}\la{energyb}
    \fr{d}{dt}\mathcal{E}+\mathcal{D}= 0   
    \end{align}
    where
    \begin{align}
    \mathcal{E} = \sum_{i=1}^N\int_\Omega c_i\ln\fr{c_i}{\bar{c}_i}\,dx+\fr{\epsilon}{2}\int_\Omega |\na\Phi|^2\,dx
    \end{align}
    and
    \begin{align}
    \mathcal{D}=\sum_{i=1}^ND_i\int_\Omega c_i|\na(\ln c_i+z_i\Phi)|^2\,dx.  
    \end{align}
    In particular, $c_1,...,c_N$ satisfy
    \begin{align}\la{energyc}
    \fr{d}{dt}\mathcal{E}+\tilde{\mathcal{D}}\le 0   
    \end{align}
    where
    \begin{align}
        \tilde{\mathcal{D}}=D\left(\sum_{i=1}^N\int_\Omega 4|\na \sqrt{c_i}|^2\,dx+\int_\Omega \fr{2}{\epsilon}\rho^2+z_i^2c_i|\na \Phi|^2\,dx\right)
    \end{align}
    and $D=\min_i D_i$.
\end{prop}

For the second term in the dissipative term $\tilde{\mathcal{D}}$, we remark that from \eqref{pois} and the fact that $\Phi$ satisfies homogeneous Neumann boundary conditions, we have the following elliptic estimate \cite{M}
\begin{align}\la{elliptic}
\fr{2}{\epsilon}\int_\Omega \rho^2\,dx=2\epsilon\int_\Omega \left(\fr{\rho}{\epsilon}\right)^2\,dx\ge C\epsilon \int_\Omega|\na\Phi|^2\,dx 
\end{align}
where $C>0$ depends only on $\Omega$. In addition, for the first term in $\tilde{\mathcal{D}}$, we have the following bounds,

\begin{prop}\la{clogsob}
For dimensions $d=2,3,4$, we have
\begin{align}
\int_\Omega c_i\ln\fr{c_i}{\bar{c}_i}\,dx\le C\int_\Omega |\na \sqrt{c_i}|^2\,dx
\end{align}
for each $i$, where $C$ depends only on $\Omega$. 
\end{prop}
\begin{proof}
This follows from Theorem \ref{logsobprop} by taking $g=c_i$ and  $d=2,3,4$, $p=4$.
\end{proof}

Now we state and prove our main result of this section:

\begin{thm}
    Under the same assumptions as Proposition \ref{energyprop}, the concentrations $c_1,...,c_N$ converge to $\bar{c}_1,...,\bar{c}_N$ in relative entropy and $\na\Phi$ converges to $0$ in $L^2$, that is,
    \be\la{expconv1}
    \mathcal{E}(t)\le \mathcal{E}(0)e^{-rt}
    \ee
    where $r>0$ depends only on the dimension, $D_i$ and $\Omega$. In particular, $c_i$ converges to $\bar{c}_i$ in $L^1$:
    \be\la{expconv2}
    \sum_{i=1}^N\frac{\|c_i(t)-\bar{c}_i\|_{L^1}^2}{\bar{c}_i}\le C_\Omega\mathcal{E}(0)e^{-rt}
    \ee
    where $C_\Omega>0$ depends only on $\Omega$.
\end{thm}

\begin{proof}
It follows from \eqref{energyc}, \eqref{elliptic}, and Proposition \ref{clogsob} that
\begin{align}
    \fr{d}{dt}\mathcal{E}+CD\mathcal{E}\le 0
\end{align}
where $C$ depends only on $\Omega$. Then, \eqref{expconv1} follows from the above inequality. Next, using the Csiszar-Kullback-Pinsker inequality (see e.g. \cite{T}), we have for each $i$,
\begin{align}
    \|c_i-\bar{c}_i\|_{L^1}^2\le 2|\Omega|\bar{c}_i\int_\Omega c_i\ln \frac{c_i}{\bar{c}_i}\,dx. 
\end{align}
Then, \eqref{expconv2} follows from \eqref{expconv1} and the above inequality.
\end{proof}

\begin{rem}
In light of this theorem, we briefly mention a closely related work, which also addresses the question of exponential convergence to equilibrium in the context of the Nernst-Planck system. In \cite{BD}, the authors establish the exponential convergence to equilibrium for the Nernst-Planck system for nonlinear no-flux boundary conditions. For this choice of boundary conditions, an analog of the dissipative equality from Proposition \ref{energyprop} holds, and the authors use logarithmic Sobolev inequalities derived in \cite{AMTU} to control the dissipative term corresponding to $\mathcal{D}$. One main difference with our work is that the logarithmic Sobolev inequality from \cite{AMTU} is obtained from a more general result that holds on Euclidean spaces (as opposed to bounded domains); and, the restriction from Euclidean spaces to bounded domains necessitates the use of a convex cut-off function, which restricts the applicability of the inequality to functions defined on uniformly convex domains (see Remark \ref{rem}). On the other hand, with our logarithmic Sobolev inequality, we only require that the domain is connected (in particular, it need not even be simply connected).

Another difference has to do with the rate of convergence $r$. In \cite{BD}, the logarithmic Sobolev inequality is used where the underlying probability measure is $e^{-z_i\Phi}/\left(\int_\Omega e^{-z_i\Phi}\,dx\right)dx,$ in analogy with the classical logarithmic Sobolev inequality for Gaussian measures. The use of this measure is what allows for the application of the inequality from \cite{AMTU}. On the other hand, this means that the rate of convergence ends up depending on pointwise lower and upper bounds on the potential $\Phi$. While uniform in time $L^\infty$ bounds on $\Phi$ are obtainable (see e.g. \cite{ci}), sharp bounds are unknown, and all known bounds depend on negative powers of the parameter $\epsilon$. In particular, these bounds are unbounded in the electroneutral (or quasi-neutral) limit of $\epsilon\to 0$. In contrast, the rate $r>0$ obtained in our theorem depends naturally on just the diffusion coefficients $D_i$ and the domain $\Omega$. In particular, this rate is independent of $\epsilon$, which is significant in the context of the study of the electroneutral limit. We point out, however, that our ability to derive this $\epsilon$-independent rate depended crucially on our choice of boundary conditions \eqref{BC}, and it is not obvious how to generalize our proof to more complex, nonlinear boundary conditions, which preserve the dissipative structure.
\end{rem}

\end{document}